\documentclass[12pt]{amsart}

\usepackage{amscd,amssymb,amsopn,amsmath,amsthm,mathrsfs,graphics,amsfonts,enumerate,verbatim,calc}

\usepackage[all]{xy}

\usepackage{color}

\usepackage[OT2,OT1]{fontenc}
\newcommand\cyr{%
\renewcommand\rmdefault{wncyr}%
\renewcommand\sfdefault{wncyss}%
\renewcommand\encodingdefault{OT2}%
\normalfont
\selectfont}
\DeclareTextFontCommand{\textcyr}{\cyr}

\usepackage{amssymb,amsmath}

\DeclareFontFamily{OT1}{rsfs}{}
\DeclareFontShape{OT1}{rsfs}{n}{it}{<-> rsfs10}{}
\DeclareMathAlphabet{\mathscr}{OT1}{rsfs}{n}{it}

\numberwithin{equation}{section}
\newtheorem{theorem}{Theorem}[section]
\newtheorem{lemma}[theorem]{Lemma}

\newtheorem{pro}[theorem]{Proposition}

\theoremstyle{definition}

\theoremstyle{remark}

\renewcommand{\ker}{\operatorname{Ker}}

\newcommand{\Hom}{\operatorname{Hom}}

\newcommand{\ann}{\operatorname{ann}}

\newcommand{\Char}{\operatorname{char}}

\newcommand{\Span}{\operatorname{span}}

\newcommand{\fkR}{\mathfrak{R}}

\newcommand{\e}{\mathfrak{e}}

\def\deg{\operatorname{deg}}

\begin{document}
\title[The existence of balanced neighborly polynomials]{The existence of balanced neighborly polynomials}

\author[N.T.T. Tam]{Nguyen Thi Thanh Tam}
%\address{	Mathematik und Informatik, Universit\"{a}t des Saarlandes, Campus E2 4,	D-66123 Saarbr\"{u}cken, Germany}
\address{Hung Vuong University, Phu Tho, Viet Nam}
\email{nguyenthithanhtam@hvu.edu.vn\\thanhtamnguyenhv@gmail.com}

%\thanks{2010 {\em Mathematics Subject Classification\/}: 13H10, t13D45, 13A15, 13H15.\\
%This work is partially supported by a fund of Vietnam National Foundation for Science
%and Technology Development (NAFOSTED) under grant number 101.04-2017.14. }


%\subjclass{13}
%\subjclass[2000]{Primary 13-XX}
%\subjclass[2000]{Primary ; Secondary}
%\date{\today \, (\printtime)}
%\date{\today}
\thanks{ %This  research was funded by Vietnam National Foundation for Science and Technology Development (NAFOSTED) under grant number $101.04-2017.14$
%The author was partially supported by a Grant of Japan Society for the Promotion of Science (JSPS).
This work was partially funded by Vingroup Joint Stock Company and supported by the PhD Scholarship Programme of Vingroup Innovation Foundation (VINIF), Vingroup Big Data Institute (VINBIGDATA)}

\begin{abstract} 
Inspired by the definition of balanced neighborly spheres, we define balanced neighborly polynomials and study the existence of these polynomials. The goal of this article is to construct balanced neighborly polynomials of type $(k,k,k,k)$ over any field $K$ for all $k \neq 2$, and show that a balanced neighborly polynomial of type $(2,2,2,2)$ exists if and only if $\Char(K) \neq 2$. Besides, we also discuss a relation between balanced neighborly polynomials and balanced neighborly simplicial spheres.
\end{abstract}

\maketitle

\section{Introduction}
Finding sharp upper bounds for the Hilbert function of certain graded algebras is
an important problem relating both commutative algebra and combinatorics. One of the most famous results on this topic is Stanley's solution of the Upper Bound Conjecture, which gives an upper bound on the number of faces of simplicial spheres by translating the problem into a problem on upper bounds of Hilbert function of graded algebras. In this paper,
we study upper bounds for the Hilbert function of Gorenstein algebras defined by certain polynomials, which we call balanced polynomials.

The motivation of this paper comes from the study of face numbers of balanced simplicial complexes. Balanced simplicial complexes are defined by Stanley \cite{St1} and their face number has been studied by
many researchers \cite{BFS, KN}. One interesting problem on balanced simplicial complexes is to find the sharp upper bound of the number of $i$-faces of balanced simplicial spheres with $n$ vertices for a fixed
integers $i$ and $n$. Recently, to study this problem, Zheng \cite{Zh} defined the notion of balanced neighborly spheres and study their existence. Balanced neighborly spheres are important ssince if a
balanced neighborly sphere of type $(k,k,...,k)$ exists then it maximizes the numbers of faces of balanced simplicial $(d - 1)$-spheres with $kd$ vertices. Especially, the paper \cite{Zh} proved that balanced neighborly spheres of type $(2, 2, 2, 2)$
do not exist, but type $(3, 3, 3, 3)$ exist. Also, Zheng asked if balanced neighborly spheres of type $(k, k, k, k)$ exist for all $k$  \cite[Question 15]{Zh}. Motivated by this work of Zheng, we define balanced neighborly polynomials (see Section 2 for the definition) and study their existence. As we will prove in Proposition 4.1, if balanced neighborly simplicial spheres of type $(n_1, \ldots, n_d)$ exist, then balanced neighborly polynomials of type $(n_1, \ldots, n_d)$ exist over any field $K.$ Thus, the study of the existence problem of balanced neighborly polynomial can be considered as an algebraic abstraction of the existence problem on balanced neighborly spheres. On the other hand, the goal of this paper is to construct a balanced neighborly polynomial of type $(k,k,k,k)$ over any field $K$ for all $k \neq 2$, (see Theorems \ref{bd1}, \ref{bd2}, and  \ref{bd3}) and show that a balanced neighborly polynomial of type $(2,2,2,2)$ exists only when $\Char(K) \neq 2$ (see Theorem \ref{3.1}).
% Our results are related to the existence problem asked by Zheng. Indeed, we proved that if balanced neighborly simplicial spheres of type $(n_1, n_2, \cdots, n_d)$ exist, then balanced neighborly polynomials of type $(n_1, n_2, \cdots, n_d)$ exist over any field $K$ in Theorem \ref{4.1}. This means that balanced neighborly polynomials can be considered as an algebraic abstraction of balanced neighborly simplicial spheres.  Thus our results answer an algebraic abstraction of the existence problem of balanced neighborly simplicial spheres of type $(k,k,k,k)$ which was studied in \cite{Zh}.
%We proved that if balanced neighborly simplicial spheres of type $(n_1, n_2, \cdots, n_d)$ exist, then balanced neighborly polynomials of type $(n_1, n_2, \cdots, n_d)$ exist over any field $K$ in Theorem \ref{4.1}.
 Moreover, the results give an alternative proof for Zheng's result proving that balanced neighborly simplicial spheres of type $(2,2,2,2)$ do not exist. 

The structure of this article is organized as follows. In section 2, we discuss basic properties of balanced neighborly polynomials. In section 3, we construct a balanced neighborly polynomial of type $(k,k,k,k)$ for any $k \neq 2$ and prove that there exists a balanced neighborly polynomial of type $(2,2,2,2)$ only when $\Char K \neq 2$.  In section 4, we discuss a relation between balanced neighborly polynomials and balanced neighborly simplicial spheres.
\section{Balanced neighborly polynomials}
 In this section we introduce notations that will be used throughout this article. 
 Let $d, n_1, \ldots, n_d$ be positive integers, and let $R:=K[x_{i j} \mid 1 \leq i \leq d, 1 \leq j \leq n_i]$ be a polynomial ring over a field $K$.
The ring $R$ has a $\mathbb{Z}^d$-graded structure by setting $\deg (x _{i j}) = \e_i$ for $1 \leq i \leq d$, where $\e_1, \ldots, \e_d$ are the unit vectors of $\mathbb{Z}^d$. 
%Then $S$ becomes a $\mathbb{Z}^d$-graded polynomial ring. 
We say that $f \in R$ is a {\it balanced polynomial }if $f$ is $\mathbb{Z}^d$-graded  and $\deg(f)=(1,1, \ldots, 1).$ For a balanced polynomial $f \in R$ and $S \subseteq [d] = \{1,2, \ldots,d \}$, we set 
$$H(f,S) = \dim_K \{ g(\partial_{i j}) f  \mid g(x_{i  j}) \in R_{\e_S}\},$$
where $\partial_{i j} = {\partial}/{\partial x_{i j}}$, $\e_S = \underset {i \in S} {\sum}\e_i \in \mathbb{Z}^d$, and $R_{\e_S}$ is the graded component of $R$ of degree $\e_S$. Moreover, a balanced polynomial $f$ is said to be {\it neighborly of type $(n_1, \ldots, n_d) \in  \mathbb{Z}^d$} if for every subset $S \subseteq [d]$ with $|S| \leq \dfrac{d}{2}$, we have $H(f,S) = \underset {k\in S} {\prod} n_k$. 

For a graded polynomial $f = f(x_1,\ldots,x_n) \in \fkR= K[x_1, \cdots, x_n],$ let 
$$\ann(f) = \{g(x_1,\ldots,x_n) \in \fkR \mid g(\partial_{1}, \cdots, \partial_n) f = 0\},$$ where 
$\partial_{i} = \dfrac{\partial}{\partial x_{i}}.$ We say that a polynomial $f$ is {\it squarefree} if $f$ is a linear combination of squarefree monomials. For $F \subseteq [n]$, let $x^F= \underset{i \in F} {\prod}x_i$. If
$$f=\sum_{F\subset [n], |F|=d} \alpha_Fx^F, \quad g=\sum_{F\subset [n], |F|=d} \beta_F x^F, \text{ where }\alpha_F, \beta_G \in K $$
are squarefree polynomials having the same degree, then $g(\partial_{1}, \cdots, \partial_n)f= \underset{F \subset [n]} {\sum} \alpha_F \beta_F$. We say that an Artinian graded $K$-algebra 
$\fkR/I = \displaystyle \bigoplus_{k=0}^d (\fkR/I)_k$, where  $(\fkR/I)_d \neq 0$ is {\it Gorenstein of socle degree $d$} if $0:_{\fkR/I} (x_1, \ldots, x_n) = (\fkR/I)_d$ has $K$-dimension $1$, where $I$ is an ideal of $\fkR$. Then we have the following property.
\begin{lemma} \label{2.1} We have
\begin{itemize}
\item [(i)] If $f$ is a squarefree polynomial of degree $d$, then $x_1^2, \ldots, x_n^2 \in \ann(f)$ and $\fkR/(\ann(f))$ is an Artinian Gorenstein graded $k$-algebra of socle degree $d$.
\item [(ii)] If $I$ is a homogenerous ideal such that $\fkR/I$ is  an Artinian Gorenstein graded $k$-algebra of socle degree $d$ and $x_1^2, \ldots, x_n^2 \in I$, then there is a squarefree polynomial such that $\ann(f)=I.$
\end{itemize}
\end{lemma}
\begin{proof}
(i) Assume $f$ is squarefree. Then $\partial_i^2f=0, \forall i =1,2, \ldots, n$. So $x_1^2, \ldots, x_n^2 \in \ann(f)$. Next, we prove that $0:_{\fkR/\ann(f)}(x_1, \ldots, x_n) = (\fkR/\ann(f))_d.$ Let $g \in \fkR$ be a homogeneous element of degree $k<d$ with $g \notin \ann(f)$. We must prove that $g $ is not an element of $0:_{\fkR/\ann(f)} (x_1, \ldots, x_n)$, that is, $x_j g \notin \ann(f)$ for some $j$. Since $g \notin \ann(f),$ we have $g(\partial_{1}, \cdots, \partial_n)f \neq 0$. Since $h=g(\partial_{1}, \cdots, \partial_n)f$ is a squarefree polynomial of degree  $\geq 1$, $\partial_j(h) \neq 0$ for some $j$. From this we obtain that $x_jg \notin \ann(f).$ Now we prove $\dim_K (\fkR/\ann(f))_d=1.$ Indeed, let 
$$f=\underset{F\subset [n], |F|=d} {\sum} \alpha_Fx^F \text{ for }\alpha_F \in K, $$ $E_k=\Span_K\{ x^F : |F|=k\}$ and $E=E_0 \oplus E_1\oplus \cdots \oplus E_k$ be the set of all squarefree polynomials at most $k$. We consider the linear map 
\begin{align*}
\psi: E_d &\to K\\
g=\sum_{F\subset [n], |F|=d} \beta_F x^F &\mapsto \sum_{F\subset [n], |F|=d} \beta_F \alpha_F = g(\partial_{1}, \cdots, \partial_n)f
\end{align*}
Then $\ker \psi = \{ g \in E_d \mid g(\partial_{1}, \cdots, \partial_n)f = 0\} = (\ann(f) \cap E)_d$. Since $\psi$ is a non-zero map, $\dim_K \ker \psi = \dim_K E_d - 1.$ Moreover, $\ann(f) = (x_1^2, \ldots,x_n^2) \oplus (\ann(f) \cap E)$ and $\fkR= (x_1^2, \ldots,x_n^2) \oplus E$. Then
\begin{align*}
 \dim_K(\fkR/\ann(f))_d & = \dim_K \fkR_d - \dim_K \ann(f)_d\\
& =\dim_K((x_1^2, \ldots,x_n^2)_d \oplus E_d)- \dim_K (x_1^2, \ldots,x_n^2)_d \oplus(\ann(f) \cap E)_d\\
&=\dim_K E_d - \dim_K \ker \psi =1.
\end{align*} 
(ii) Suppose $\fkR/I$ is Artinian Gorenstein of socle degree $d$ and $x_1^2, \ldots,x_n^2 \in I$. Then $I_d=(x_1^2, \ldots,x_n^2)_d \bigoplus (E\cap I)_d$ as $K$-vector space and 
\begin{align*}
\dim_K(E\cap I)_d &= \dim_K I_d - \dim_K (x_1^2, \ldots,x_n^2)_d\\
&=(\dim_K R_d - 1) - \dim_K (x_1^2, \ldots,x_n^2)_d\\
&=\dim_K E_d -1.
\end{align*}
Since $(E\cap I)_d$ is a codimension $1$ subspace of $E_d$, it can be written as a kernel of a $K$-linear function, that is, 

$$(E \cap I)_d = \left(\Big \{ \sum_{F\subset [n], |F|=d} \beta_F x^F \in E_d \mid \sum_{F\subset [n], |F|=d} \alpha_F \beta_F =0 \Big \}\right)$$
for some $\alpha_F \in K$ such that $|F|=d.$
Let $f = \underset {F\subset [n], |F|=d} {\sum} \alpha_F x^F$. We claim that $\ann(f) = I$. Indeed, 
$$(E\cap I)_d = \left(\Big \{\sum_{F\subset [n], |F|=d} \beta_F  x^F \in E_d \mid \sum_{F\subset [n], |F|=d} \alpha_F \beta_F =g(\partial_{i j})f= 0 \Big \}\right)=(\ann(f) \cap E)_d.$$
In particular, 
$$I_d = (x_1^2, \ldots,x_n^2)_d+ (E \cap I)_d = (x_1^2, \ldots,x_n^2)_d+ (E\cap \ann(f))_d=\ann(f)_d.$$
Note that for any ideal $J$ of $\fkR$, if $\fkR/J$ is an Artinian Gorenstein algebra of socle degree $d$, then $J_d$ determines the whole $J$. In fact, for any polynomial $g \in \fkR$ of degree $t < d$, one has $g \in J $ if and only if $(x_1, \ldots, x_n)^{d-t} g \subset J_d$. %jm(because if $g \notin J,$ then $\exists h \in S_{d-t}$ such that $h.g$ in non-zero in $S/J$). 
Since both $\fkR/I$ and $\fkR/\ann(f)$ are Artinian Gorenstein algebras of socle degree $d$, $I_d = \ann(f)_d$ implies $I = \ann(f)$.
\end{proof}

From the definition of $H(f,S)$ and $\ann(f)$ as above, it is easy to prove that $H(f,S)$ has the following symmetry.
\begin{lemma} \label {2.2} If $f \in \fkR$ is a balanced polynomial of degree $d$, then for any $S\subseteq [d]$
\begin{itemize}
\item [(i)] $H(f,S)=\dim_K (\fkR/\ann(f))_{\e_S}$,
\item [(ii)]$H(f,S) = H(f, [d] \backslash S)$. 
\end{itemize}
\end{lemma}
\begin{proof}
(i) Set $V_{\e_S}(f) = \{ g(\partial_{1}, \cdots, \partial_n)f  \mid g(x_1,\ldots, x_n) \in \fkR_{\e_S}\}$. Then $H(f,S) = \dim_K V_{\e_S} (f).$ Consider the short exact sequence 
$$\xymatrix{0\ar[r]& \ker \varphi \ar[r] & \fkR_{\e_S} \ar[r]^\varphi & V_{\e_S} (f) \ar[r]&0},$$
where $\varphi(g) = g(\partial_{i j})f$ for all $g \in \fkR_{\e_S}$. Note that $\ker \varphi = \ann(f)_{\e_S}$.
Therefore $$H(f,S) = \dim_K \fkR_{\e_S} - \dim_K \ann(f)_{\e_S} = \dim_K (\fkR/\ann(f))_{\e_S}.$$
(ii) Set $A=\fkR/\ann(f)=\displaystyle \bigoplus_{k=1}^d A_k.$ Since $\fkR/\ann(f)$ is a Gorenstein ring. The bilinear map $$A_{\e_A}\times A_{\e_{[d]\backslash A}} \to A_{\e_{[d]}} \cong K$$ is defined by the multiplication non-degenerate. Therefore, $A_{\e_S} \cong \Hom_K(A_{\e_{[d]\backslash S}},K) \cong A_{\e_{[d]\backslash S}},$ and the statement follows from (i).
\end{proof}
In the next section, we study balanced neighborly polynomials of degree $4$. For this, the next Lemma says that we only need to consider balanced neighborly polynomials of type $(k,k,k,k)$.
\begin{lemma} \label{2.3}
  Let $f$ be a balanced neighborly polynomial of type $(n_1, \cdots, n_d) \in  \mathbb{Z}^d$. If $d$ is even then $n_1 =n_2= \cdots = n_d.$ 
\end{lemma}
\begin{proof}
By the definition of a balanced polynomial $f \in \fkR$, we see that the $\mathbb{Z}^d$-degree of $f$ is $(1,1, \ldots,1)$. Since $d$ is even, we set $d = 2m$. Assume that $n_1 \geq n_2 \geq \cdots \geq n_d$. Then by Lemma \ref{2.2} (ii), we have
$$n_1 n_2  \cdots n_m = H(f,\{1,2, \cdots, m\}) = H(f,\{m+1, \ldots, 2m\})=n_{m+1} n_{m+2} \ldots n_{2m}.$$
Since $n_1 \geq n_2 \geq \cdots \geq n_d$,  we conclude that $n_1 = n_2 = \cdots = n_m = n_{m+1} = \cdots = n_{2m}.$
\end{proof}
In case $d$ is even, Lemma \ref{2.3} for balanced simplicial spheres was proved by Zheng in \cite{Zh}.
\section{The existence of balanced neighborly polynomials of type $(k,k,k,k)$}
Throughout this section, we consider ring $$R = K[x_1, \ldots, x_k, y_1, \ldots, y_k, z_1, \ldots, z_k, w_1, \ldots, w_k],$$ where $\deg (x_i) = \e_1=(1,0,0,0)$, $\deg (y_i) =\e_2= (0,1,0,0)$, $\deg (z_i) = \e_3=(0,0,1,0)$, $\deg (w_i) =\e_4= (0,0,0,1)$. So, if $f \in R $ is  a balanced polynomial, then we may assume that
\begin{align*}
\label{CT1}
f =\sum_{(i_1, i_2, i_3, i_4)\in { \{1,2,\ldots,k\}}^4} a_{i_1, i_2, i_3, i_4} x_{i_1} y_{i_2} z_{i_3} w_{i_4}, \quad
\text{where } a_{i_1, i_2, i_3, i_4} \in K.
\end{align*}
Therefore, 
\begin{align*}
H(f,\{1,2\})&= H(f,\{ 3,4\})= \dim \Span_K \left\lbrace  \frac{\partial}{\partial x_i}\frac{\partial}{\partial y_j}f \mid 1 \leq i, j \leq k \right\rbrace \\
H(f,\{ 1,3\})&= H(f,\{ 2,4\})= \dim \Span_K \left\lbrace  \frac{\partial}{\partial x_i}\frac{\partial}{\partial z_j}f \mid 1 \leq i, j \leq k\right\rbrace  \\
H(f,\{ 1,4\})&= H(f,\{ 2,3\})= \dim \Span_K \left\lbrace  \frac{\partial}{\partial x_i}\frac{\partial}{\partial w_j}f \mid 1 \leq i, j \leq k\right\rbrace.
\end{align*}
\begin{theorem}\label{3.1}	
There exists a balanced neighborly polynomial of type $(2,2,2,2)$ if and only if $\Char K \neq 2$. 
\end{theorem}
\begin{proof}
We first prove that if  $\Char K \neq 2$ then there exists a balanced neighborly polynomial of type $(2,2,2,2)$. Let
%let $S=k[x_1, x_2, y_1, y_2, z_1,z_2,w_1,w_2]$ with $\deg (x_i) =(1,0,0,0), \deg(y_i) = (0,1,0,0), \deg(z_i) = (0,0,1,0)$ and $\deg(w_i) = (0,0,0,1)$.\\
 $$f = -x_1y_1z_2w_2 -x_1y_2z_1w_1 + x_1y_2z_2w_2-x_2y_1z_1w_2 +x_2y_1z_2w_1 +x_2y_2z_2w_1.$$
 We have 
\begin{align*}
H(f,\{ 1,2\})  =&  \dim_K \Span _K \{-z_2w_2, -z_1w_1+ z_2w_2, -z_1w_2 +z_2w_1, z_2w_1 \},\\
 H(f,\{ 1,3\}) = & \dim_K \Span _K \{-y_2w_1, y_2w_2-y_1w_2, -y_1w_2, y_2w_1+ y_1w_1 \},\\
H(f,\{ 1,4\})= & \dim_K \Span _K \{-y_2z_1 , y_2z_2-y_1z_2, y_2z_2 +y_1 z_2, -y_1z_1 \}.
\end{align*} 
 Then  $$H(f,\{ 1,2\})=H(f,\{ 1,3\})=H(f,\{ 1,4\})=4.$$
 So, $f$ is a balanced neighborly polynomial of type $(2,2,2,2)$.
Conversely, suppose that $\Char K = 2$ and let
\begin{align*}
 f &= a_1x_1y_1z_1w_1 + a_2x_1y_1z_1w_2 + a_3x_1y_1z_2w_1 + a_4x_1y_1z_2w_2 + a_5 x_1y_2z_1w_1 + a_6 x_1y_2z_1w_2 \\
 &+ a_7x_1y_2z_2w_1 + a_8x_1y_2z_2w_2 + a_9 x_2y_1z_1w_1 + a_{10} x_2y_1z_1w_2 + a_{11}x_2y_1z_2w_1 + a_{12}x_2y_1z_2w_2\\ &+ a_{13} x_2y_2z_1w_1 + a_{14} x_2y_2z_1w_2 + a_{15}x_2y_2z_2w_1 + a_{16}x_2y_2z_2w_2 ,
 \end{align*}
where $a_i \in K, i = 1,2, \ldots, 16.$ 
We prove that $f$ is not a balanced neighborly polynomial of type $(2,2,2,2)$. Let 
$$M =\begin{pmatrix}
a_1 & a_2 & a_3 & a_4 \\
a_5 & a_6 & a_7 & a_8 \\
a_9 & a_{10} & a_{11} & a_{12} \\
a_{13} & a_{14} & a_{15} & a_{16} \\
\end{pmatrix}. $$
Since $H(f,\{ 1,2\})$ is the $K$-dimension of 
\begin{align*}
 \Span _K & \{a_1z_1w_1 + a_2z_1w_2 + a_3z_2w_1 + a_4z_2w_2,  a_5 z_1w_1 + a_6 z_1w_2 + a_7z_2w_1 + a_8z_2w_2, \\
 & a_9 z_1w_1 + a_{10} z_1w_2 + a_{11}z_2w_1 + a_{12}z_2w_2 ,  a_{13} z_1w_1 + a_{14} z_1w_2 + a_{15}z_2w_1 + a_{16}z_2w_2 \},
\end{align*}
$H(f,\{ 1,2\})=4$ if and only if $M$ is a nonsingular matrix.
Let
$$ 
N =\begin{pmatrix}
a_1 & a_2 & a_5 & a_6 \\
a_3 & a_4 & a_7 & a_8 \\
a_9 & a_{10} & a_{13} & a_{14} \\
a_{11} & a_{12} & a_{15} & a_{16} \\
\end{pmatrix} \text{ and }
P =\begin{pmatrix}
a_1 & a_3 & a_5 & a_7 \\
a_2 & a_4 & a_6 & a_8 \\
a_9 & a_{11} & a_{13} & a_{15} \\
a_{10} & a_{12} & a_{14} & a_{16} \\
\end{pmatrix}.
$$
A similar computation shows that $N$ (resp. $P$) is a non-singular matrix if and only if $ H(f,\{ 1,3\}) = 4$ (resp. $ H(f,\{ 1,4\}) = 4$). Moreover, we have $\det(M) - \det(N) =\det(P).$ It follows that one of $\det (M), \det (N),$ and $ \det (P)$ must be zero. So, $f$ is not a balanced neighborly polynomial of type $(2,2,2,2)$.
\end{proof}
For an integer $i$, let $[i]_k \in \{ 1,2,. \ldots, k \} $ be the number such that 
$[i]_k \equiv i(\text{mod} $ $k).$ 
%Set $T: =\{ z_i w_j \mid i, j = 1, 2, \ldots, k \}$. It is obvious that $|T|=k^2.$
\begin{theorem} \label{bd1}
Let $k \in \mathbb{N}$ be odd and $f = \sum\limits_{1 \leq i,j \leq k} x_iy_jz_{{[j-i]}_k} w_{{[i+j]}_k}.$ Then $f$ is a balanced neighborly polynomial of type $(k,k,k,k).$ 
\end{theorem}
\begin{proof}
 Set $A = \{ z_{[j-i]_k} w_{[j+i]_k} \mid i ,j = 1,2, \ldots, k \}$. We consider the map 
 \begin{align*}
 \varphi: \{ 1,2, \ldots, k\} ^2 &\rightarrow  \{ 1,2, \ldots, k\} ^2 \\
 (i,j) &\mapsto ([j-i]_k, [j+i]_k).
 \end{align*}
We prove that $\varphi$ is bijective. It suffices to prove the injectivity.  Suppose $([j-i]_k, [j+i]_k) = ([j'-i']_k, [j'+i']_k)$, where $(i,j), (i',j') \in  \{ 1,2, \ldots, k\} ^2 $. Then we have 
$j-i \equiv j'-i' (\text{mod} $ $k)$,  $j+i \equiv j'+i' (\text{mod} $ $k)$. It follows that $2j \equiv 2j'(\text{mod} $ $k)$ and $2i \equiv 2i'(\text{mod} $ $k)$. Since $k$ is odd,  $j \equiv j'(\text{mod} $ $k)$ and $i \equiv i'(\text{mod} $ $k)$. Then $(i,j) = (i',j')$ so, $\varphi$ is an injective.
%Besides, suppose that $(s,t) =  ([j-i]_k, [j+i]_k) \in \{ 1,2, \ldots, k\} ^2 $. We have $j-i \cong s(\mod k)$,  $j+i \cong t(\mod k)$ and then $j \cong \frac{s+t}{2} (\mod k)$, $i \cong \frac{t-s}{2} (\mod k)$. Note that $([\frac{t-s}{2}]_k,[\frac{s+t}{2}]_k) \in \{ 1,2, \ldots, k\} ^2$ since $k$ is an odd number.  So, for all $s, t \in \{ 1,2, \ldots, k\}$, there exist $(i,j)=([\frac{t-s}{2}]_k,[\frac{s+t}{2}]_k) \in \{ 1,2, \ldots, k\} ^2 $ such that $\varphi(i,j) = (s,t)$. Then $\varphi$ is an bijection.  Then, we have $|A| = k^2$. On the other hand $A \subset T$,  so $A = T.$ It follows that each components of system $A$ are different. Therefore, system of elements $A$ are linearly independent. Then
The bijectivity of $\varphi$ implies that $A = \{ z_i w_j \mid i ,j = 1,2, \ldots, k \}$. Hence
$$H(f,\{ 1,2\})  =\dim_K \Span_K A = k^2.$$
In order to compute $\dim_K R_{\e_1+\e_3} $ and $\dim_K R_{\e_1+\e_4}$, we rewrite $f$ as follows 
$$f = \sum\limits_{1 \leq i,j \leq k} x_iy_{[i+j]_k}z_j w_{{[2i+j]}_k} = \sum\limits_{1 \leq i,j \leq k} x_iy_{[j-i]_k}z_{[j-2i]_k} w_j.$$
By a similar argument as computing $H(f,\{ 1,2\})$, we have 
$$H(f,\{ 1,3\}) = \dim_K \Span _K \{ y_{[i+j]_k} w_{{[2i+j]}_k} \mid i,j=1 , \ldots, k \} = k^2, $$
$$H(f,\{ 1,4\}) = \dim_K \Span _K \{ y_{[j-i]_k}z_{{[j-2i]}_k} \mid i,j=1,\ldots, k \} = k^2.$$
So, $f$ is a balanced neighborly polynomial of type $(k,k,k,k)$.
\end{proof}
From now on we assume that $k$ is even. Let $\overline{x}_i = x_{i+ \frac{k}{2}}, $ $ \overline{y}_i = y_{i+ \frac{k}{2}}$, $\overline{z}_{i = z_i+ \frac{k}{2}}$, $\overline{w}_i = w_{i+ \frac{k}{2}}$ for $i = 1,2, \ldots, \frac{k}{2}.$ 
\begin{theorem} \label{bd2}
Suppose $k =4m,$ with $ m \in \mathbb{N}$. Let
\begin{align*}
f &= \sum\limits_{1 \leq i,j \leq 2m, i: odd} x_iy_jz_{{[i+j-1]}_{2m}} w_{{[\frac{i-1}{2} +j]}_{2m}} 
   &  + \sum\limits_{1 \leq i,j \leq 2m, i: odd} \overline{x}_i y_j \overline{z}_{{[i+j-1]}_{2m}} \overline{w}_{{[\frac{i-1}{2} +j+m]}_{2m}} \\
&+  \sum\limits_{1 \leq i,j \leq 2m, i: odd} x_i\overline{y}_j\overline{z}_{{[i+j-1]}_{2m}} \overline{w}_{{[\frac{i-1}{2}+j]}_{2m}}
&+  \sum\limits_{1 \leq i,j \leq 2m, i: odd} \overline{x}_i\overline{y}_jz_{{[i+j-1]}_{2m}} w_{{[\frac{i-1}{2}+j+m]}_{2m}} \\
& +  \sum\limits_{1 \leq i,j \leq 2m, i: even} x_iy_jz_{{[i+j-1]}_{2m}} \overline{w}_{{[\frac{i-2}{2}+j]}_{2m}}
&+  \sum\limits_{1 \leq i,j \leq 2m, i: even} \overline{x}_iy_j\overline{z}_{{[i+j-1]}_{2m}} w_{{[\frac{i-2}{2}+j + m]}_{2m}}\\
&+  \sum\limits_{1 \leq i,j \leq 2m, i: even} x_i\overline{y}_j\overline{z}_{{[i+j-1]}_{2m}} w_{{[\frac{i-2}{2}+j]}_{2m}}
&+  \sum\limits_{1 \leq i,j \leq 2m, i: even} \overline{x}_i\overline{y}_jz_{{[i+j-1]}_{2m}} \overline{w}_{{[\frac{i-2}{2}+j+m]}_{2m}}.
\end{align*}
 Then $f$ is a balanced neighborly polynomial of type $(k,k,k,k).$ 
\end{theorem}
\begin{proof}
Set 
\begin{align*}
B_1: =&\{z_{{[i+j-1]}_{2m}} w_{{[\frac{i-1}{2} +j]}_{2m}},  \overline{z}_{{[i+j-1]}_{2m}} \overline{w}_{{[\frac{i-1}{2}+j]}_{2m}} \mid i,j = 1, 2, \ldots, 2m, i: odd\},\\
B'_1: =&\{\overline{z}_{{[i+j-1]}_{2m}} \overline{w}_{{[\frac{i-1}{2} +j+m]}_{2m}} , z_{{[i+j-1]}_{2m}} w_{{[\frac{i-1}{2}+j+m]}_{2m}}\mid i,j = 1, 2, \ldots, 2m, i: odd\},\\
B_2:=& \{ z_{{[i+j-1]}_{2m}} \overline{w}_{{[\frac{i-2}{2}+j]}_{2m}}, \overline{z}_{{[i+j-1]}_{2m}} w_{{[\frac{i-2}{2}+j]}_{2m}} \mid i,j = 1, 2, \ldots, 2m, i: even\}, \\
B'_2:=& \{  \overline{z}_{{[i+j-1]}_{2m}} w_{{[\frac{i-2}{2}+j + m]}_{2m}}, z_{{[i+j-1]}_{2m}} \overline{w}_{{[\frac{i-2}{2}+j+m]}_{2m}}  \mid i,j = 1, 2, \ldots, 2m, i: even\}. 
\end{align*}
% Note that, $B_1 \cap B_2 = \emptyset$.  It is clear that $B_1 \cup B_2 \subset T $ and $|B_1 \cup B_2| = |T| = (4m)^2 = k^2$. It follows that $B_1 \cup B_2=T$.\\
Then $B_1$, $B'_1$, $B_2$ and $B'_2$ are pairwise disjoint. Hence
 $$ |B_1 \cup B'_1 \cup B_2  \cup B'_2 | =|B_1|+ | B'_1| + |B_2| +| B'_2 |.$$
 In the same way as the proof of Theorem \ref{bd1}, one can prove the injectivity of the following maps:
\begin{align*}
\{ 1,3, \ldots, 2m-1\} \times \{ 1,2, \ldots, 2m\} &\to \{ 1,2, \ldots, 2m\}^2\\
 (i,j) &\mapsto ([i+j-1]_{2m}, [\frac{i-1}{2}+j]_{2m}),\\
 (i,j) &\mapsto ([i+j-1]_{2m}, [\frac{i-1}{2}+j+m]_{2m})
 \end{align*}
 and
\begin{align*}
\{ 2,4, \ldots, 2m\} \times \{ 1,2, \ldots, 2m\} &\to \{ 1,2, \ldots, 2m\}^2\\
 (i,j) &\mapsto ([i+j-1]_{2m}, [\frac{i-2}{2}+j]_{2m}),\\
 (i,j) &\mapsto ([i+j-1]_{2m}, [\frac{i-2}{2}+j+m]_{2m}).
\end{align*}
It follows that $|B_1| = |B'_1|= |B_2| = |B'_2| = 4m^2$. Therefore
$$ H(f,\{ 1,2\}) =\dim_K \Span _K (B_1 \cup B'_1 \cup B_2  \cup B'_2)= (4m)^2 = k^2.$$
In order to compute $\dim_K R_{\e_1 + \e_3}$, we can rewrite $f$ as follows \begin{align*}
f &= \sum\limits_{1 \leq i,j \leq 2m, i: odd} x_iy_{{[j+1-i]}_{2m}} z_jw_{{[j-\frac{i-1}{2}]}_{2m}} 
   &  + \sum\limits_{1 \leq i,j \leq 2m, i: odd} \overline{x}_i y_{[j+1-i]_{2m}} \overline{z}_j \overline{w}_{{[j-\frac{i-1}{2} +m]}_{2m}} \\
&+  \sum\limits_{1 \leq i,j \leq 2m, i: odd} x_i\overline{y}_{[j+1-i]_{2m}}\overline{z}_j \overline{w}_{{[j-\frac{i-1}{2}]}_{2m}}
&+  \sum\limits_{1 \leq i,j \leq 2m, i: odd} \overline{x}_i\overline{y}_{[j+1-i]_{2m}}z_j w_{{[j-\frac{i-1}{2}+m]}_{2m}} \\
& +  \sum\limits_{1 \leq i,j \leq 2m, i: even} x_iy_{[j+1-i]_{2m}}z_j \overline{w}_{{[j-\frac{i}{2}]}_{2m}}
&+  \sum\limits_{1 \leq i,j \leq 2m, i: even} \overline{x}_iy_{[j+1-i]_{2m}}\overline{z}_j w_{{[j-\frac{i}{2} + m]}_{2m}}\\
&+  \sum\limits_{1 \leq i,j \leq 2m, i: even} x_i\overline{y}_{[j+1-i]_{2m}}\overline{z}_j w_{{[j-\frac{i}{2}]}_{2m}}
&+  \sum\limits_{1 \leq i,j \leq 2m, i: even} \overline{x}_i\overline{y}_{[j+1-i]_{2m}}z_j \overline{w}_{{[j-\frac{i}{2}+m]}_{2m}}.
\end{align*}
Set
\begin{align*}
B_3:=&\{ y_{{[j+1-i]}_{2m}} w_{{[j-\frac{i-1}{2}]}_{2m}},  \overline{y}_{{[j+1-i]}_{2m}} \overline{w}_{{[j-\frac{i-1}{2}]}_{2m}}  \mid i,j = 1, 2, \ldots, 2m, i: odd\},\\
B'_3:=&\{y_{{[j+1-i]}_{2m}} \overline{w}_{{[j-\frac{i-1}{2}+m]}_{2m}},  \overline{y}_{{[j+1-i]}_{2m}} w_{{[j-\frac{i-1}{2}+m]}_{2m}} \mid i,j = 1, 2, \ldots, 2m, i: odd\},\\
 B_4:= & \{ y_{{[j+1-i]}_{2m}} \overline{w}_{{[j-\frac{i}{2}]}_{2m}}, \overline{y}_{{[j+1-i]}_{2m}} w_{{[j-\frac{i}{2}]}_{2m}}  \mid i,j = 1, 2, \ldots, 2m, i: even \},\\
  B'_4:= & \{  y_{{[j+1-i]}_{2m}} w_{{[j-\frac{i}{2}+ m]}_{2m}}, \overline{y}_{{[j+1-i]}_{2m}} \overline{w}_{{[j-\frac{i}{2}+m]}_{2m}}  \mid i,j = 1, 2, \ldots, 2m, i: even \}.
\end{align*}
Then $H(f,\{ 1,3\}) = \dim \Span _K (B_3 \cup B'_3 \cup B_4 \cup B'_4)  $.
A similar argument as computing $H(f,\{ 1,2\})$, we have $H(f,\{ 1,3\}) =  (4m)^2= k^2.$
Similarly, we can rewrite $f$ as follows \begin{align*}
f &= \sum\limits_{1 \leq i,j \leq 2m, i: odd} x_iy_{[j-\frac{i-1}{2}]_{2m}} z_{[j+\frac{i-1}{2}]_{2m}}w_j 
   &  + \sum\limits_{1 \leq i,j \leq 2m, i: odd} \overline{x}_i y_{[j-\frac{i-1}{2}-m]_{2m}} \overline{z}_{[j+\frac{i-1}{2}-m]_{2m}} \overline{w}_j \\
&+  \sum\limits_{1 \leq i,j \leq 2m, i: odd} x_i\overline{y}_{[j-\frac{i-1}{2}]_{2m}}\overline{z}_{[j+\frac{i-1}{2}]_{2m}} \overline{w}_j
&+  \sum\limits_{1 \leq i,j \leq 2m, i: odd} \overline{x}_i\overline{y}_{[j-\frac{i-1}{2}-m]_{2m}}z_{[j+\frac{i-1}{2}-m]_{2m}} w_j \\
& +  \sum\limits_{1 \leq i,j \leq 2m, i: even} x_iy_{[j-\frac{i-2}{2}]_{2m}}z_{[j+\frac{i}{2}]_{2m}} \overline{w}_j
&+  \sum\limits_{1 \leq i,j \leq 2m, i: even} \overline{x}_iy_{[j-\frac{i-2}{2}-m]_{2m}}\overline{z}_{[j+\frac{i}{2}-m]_{2m}} w_j\\
&+  \sum\limits_{1 \leq i,j \leq 2m, i: even} x_i\overline{y}_{[j-\frac{i-2}{2}]_{2m}}\overline{z}_{[j+\frac{i}{2}]_{2m}} w_j
&+  \sum\limits_{1 \leq i,j \leq 2m, i: even} \overline{x}_i\overline{y}_{[j-\frac{i-2}{2}-m]_{2m}}z_{[j+\frac{i}{2}-m]_{2m}} \overline{w}_j.
\end{align*}

By similar computing as above, we get that $ H(f,\{ 1,4\}) = (4m)^2= k^2.$ So, $f$ is a balanced neighborly polynomial of type $(k,k,k,k)$.
\end{proof}
\begin{theorem} \label{bd3}
Suppose $k =4m + 2,$ with $ m \in \mathbb{N}$. Let 
\begin{align*}
f &= \sum\limits_{1 \leq i,j \leq 2m+1, i: odd} x_iy_jz_{{[i+j-1]}_{2m+1}} w_{{[\frac{i-1}{2} +j]}_{2m+1}} 
   &  + \sum\limits_{1 \leq i,j \leq 2m+1, i: odd} \overline{x}_i y_j \overline{z}_{{[i+j-1]}_{2m+1}} \overline{w}_{{[\frac{i-1}{2} +j+m]}_{2m+1}} \\
&+  \sum\limits_{1 \leq i,j \leq 2m+1, i: odd} x_i\overline{y}_j\overline{z}_{{[i+j-1]}_{2m+1}} \overline{w}_{{[\frac{i-1}{2}+j]}_{2m+1}}
&+  \sum\limits_{1 \leq i,j \leq 2m+1, i: odd} \overline{x}_i\overline{y}_jz_{{[i+j-1]}_{2m+1}} w_{{[\frac{i-1}{2}+j+m]}_{2m+1}} \\
& +  \sum\limits_{1 \leq i,j \leq 2m+1, i: even} x_iy_jz_{{[i+j-1]}_{2m+1}} \overline{w}_{{[\frac{i-2}{2}+j]}_{2m+1}}
&+  \sum\limits_{1 \leq i,j \leq 2m+1, i: even} \overline{x}_iy_j\overline{z}_{{[i+j-1]}_{2m+1}} w_{{[\frac{i}{2}+j + m]}_{2m+1}}\\
&+  \sum\limits_{1 \leq i,j \leq 2m+1, i: even} x_i\overline{y}_j\overline{z}_{{[i+j-1]}_{2m+1}} w_{{[\frac{i-2}{2}+j]}_{2m+1}}
&+  \sum\limits_{1 \leq i,j \leq 2m+1, i: even} \overline{x}_i\overline{y}_jz_{{[i+j-1]}_{2m+1}} \overline{w}_{{[\frac{i}{2}+j+m]}_{2m+1}}\\
& + \sum\limits_{1 \leq j \leq 2m+1} \overline{x}_{2m+1} y_j \overline{z}_j w_j + \sum\limits_{1 \leq j \leq 2m+1}\overline{x}_{2m+1}\overline{y}_jz_j \overline{w}_j .
\end{align*}
 Then $f$ is a balanced neighborly polynomial of type $(k,k,k,k).$ 
\end{theorem}
\begin{proof}
Set 
\begin{align*}
C_1: = & \{ z_{{[i+j-1]}_{2m+1}} w_{{[\frac{i-1}{2} +j]}_{2m+1}}  ,\overline{z}_{{[i+j-1]}_{2m+1}} \overline{w}_{{[\frac{i-1}{2}+j]}_{2m+1}}  \mid i,j = 1, 2, \ldots, 2m+1, i: odd\} ,\\
  C'_1: = & \{ \overline{z}_{{[i+j-1]}_{2m+1}} \overline{w}_{{[\frac{i-1}{2} +j+m]}_{2m+1}} , z_{{[i+j-1]}_{2m+1}} w_{{[\frac{i-1}{2}+j+m]}_{2m+1}}   \mid i,j = 1, 2, \ldots, 2m+1, i: odd, i \neq 2m+1\}, \\
  C_2:=& \{ z_{{[i+j-1]}_{2m+1}} \overline{w}_{{[\frac{i-2}{2}+j]}_{2m+1}} , \overline{z}_{{[i+j-1]}_{2m+1}} w_{{[\frac{i-2}{2}+j]}_{2m+1}} \mid i,j = 1, 2, \ldots, 2m+1, i: even\},\\
   C'_2:=& \{  \overline{z}_{{[i+j-1]}_{2m+1}} w_{{[\frac{i}{2}+j + m]}_{2m+1}} , z_{{[i+j-1]}_{2m+1}} \overline{w}_{{[\frac{i}{2}+j+m]}_{2m+1}}  \mid i,j = 1, 2, \ldots, 2m+1, i: even\},\\
 C_3:=& \{ \overline{z}_j w_j + \overline{z}_{[2m+j]_{2m+1}}\overline{w}_{[2m+j]_{2m+1}} ,z_j \overline{w}_j + z_{[2m+j]_{2m+1}}w_{[2m+j]_{2m+1}}\mid j = 1, 2, \ldots, 2m+1 \}, \\
  C'_3:=& \{ \overline{z}_j w_j  ,z_j \overline{w}_j \mid j = 1, 2, \ldots, 2m+1 \}. 
\end{align*}
  We can see that $\overline{z}_{[2m+j]_{2m+1}} \overline{w}_{[2m+j]_{2m+1}} ,z_{[2m+j]_{2m+1}}w_{[2m+j]_{2m+1}}$ belong to $C_1$. Thus
 $$\Span (C_1 \cup C'_1 \cup C_2 \cup C'_2 \cup C_3) = \Span (C_1 \cup C'_1 \cup C_2 \cup C'_2 \cup C'_3).$$
   Moreover, $\overline{z}_j w_j,z_j \overline{w}_j$ do not belong to  $C_2 \cup C'_2 $, so $C_1, C'_1, C_2, C'_2, C'_3 $ are pairwise disjoint. Hence $|C_1 \cup C'_1 \cup C_2 \cup C'_2 \cup C'_3| = |C_1|+ |C'_1| + |C_2|+ |C'_2|+ |C'_3|$.
In the same way as in the proof of Theorem \ref{bd2}, one can show
  $| C_1| = 2(m+1)(2m+1)$, $| C'_1| = 2m(2m+1)$, $|C_2|=|C'_2| = 2m(2m+1)$ and $|C'_3| = 2(2m+1)$. Then $|C_1 \cup C'_1 \cup C_2 \cup C'_2 \cup C'_3 | = (4m+2)^2= k^2$.   Therefore,
  $$ H(f,\{ 1,2\})= \dim_K \Span _K C_1 \cup C'_1 \cup C_2 \cup C'_2 \cup C'_3 = (4m+2)^2 = k^2.$$ 
On the other hand, we can rewrite $f$ as follows 
\begin{align*}
f &=  \sum\limits_{1 \leq i,j \leq 2m+1, i: odd} x_iy_{{[j+1-i]}_{2m+1}} z_jw_{{[j-\frac{i-1}{2}]}_{2m+1}} 
   &  + \sum\limits_{1 \leq i,j \leq 2m+1, i: odd} \overline{x}_i y_{[j+1-i]_{2m+1}} \overline{z}_j \overline{w}_{{[j-\frac{i-1}{2} +m]}_{2m+1}} \\
&+  \sum\limits_{1 \leq i,j \leq 2m+1, i: odd} x_i\overline{y}_{[j+1-i]_{2m+1}}\overline{z}_j \overline{w}_{{[j-\frac{i-1}{2}]}_{2m+1}}
&+  \sum\limits_{1 \leq i,j \leq 2m+1, i: odd} \overline{x}_i\overline{y}_{[j+1-i]_{2m+1}}z_j w_{{[j-\frac{i-1}{2}+m]}_{2m+1}} \\
& +  \sum\limits_{1 \leq i,j \leq 2m+1, i: even} x_iy_{[j+1-i]_{2m+1}}z_j \overline{w}_{{[j-\frac{i}{2}]}_{2m+1}}
&+  \sum\limits_{1 \leq i,j \leq 2m+1, i: even} \overline{x}_iy_{[j+1-i]_{2m+1}}\overline{z}_j w_{{[j-\frac{i-2}{2} + m]}_{2m+1}}\\
&+  \sum\limits_{1 \leq i,j \leq 2m+1, i: even} x_i\overline{y}_{[j+1-i]_{2m+1}}\overline{z}_j w_{{[j-\frac{i}{2}]}_{2m+1}}
&+  \sum\limits_{1 \leq i,j \leq 2m+1, i: even} \overline{x}_i\overline{y}_{[j+1-i]_{2m+1}}z_j\overline{w}_{{[j-\frac{i-2}{2}+m]}_{2m+1}}\\
& + \sum\limits_{1 \leq j \leq 2m+1} \overline{x}_{2m+1} y_j \overline{z}_j w_j + \sum\limits_{1 \leq j \leq 2m+1}\overline{x}_{2m+1}\overline{y}_jz_j \overline{w}_j .
\end{align*}
Set
\begin{align*}
C_4:=&\{ y_{[j+1-i]_{2m+1}} w_{{[j+\frac{1-i}{2} ]}_{2m+1}}, \overline{y}_{[j+1-i]_{2m+1}} \overline{w}_{{[j+\frac{1-i}{2}]}_{2m+1}}  \mid i,j = 1, 2, \ldots, 2m +1, i: odd\} ,\\
  C'_4:=&\{ y_{[j+1-i]_{2m+1}} \overline{w}_{[\frac{1-i}{2} +j+m]_{2m+1}}, \overline{y}_{[j+1-i]_{2m+1}} w_{{[\frac{1-i}{2}+j+m]}_{2m}}  \mid i,j = 1, 2, \ldots, 2m +1, i: odd, i \neq 2m+1\} ,\\
C_5: =& \{ y_{[j+1-i]_{2m+1}} \overline{w}_{{[j-\frac{i}{2}]}_{2m+1}},\overline{y}_{[j+1-i]_{2m+1}} w_{{[j-\frac{i}{2}]}_{2m+1}} \mid i,j = 1, 2, \ldots, 2m +1, i: even\} ,\\
 C'_5: =& \{   y_{[j+1-i]_{2m+1}} w_{{[j-\frac{i-2}{2}+m]}_{2m+1}}, \overline{y}_{[j+1-i]_{2m+1}} \overline{w}_{{[j-\frac{i-2}{2}+m]}_{2m+1}} \mid i,j = 1, 2, \ldots, 2m +1, i: even\} ,\\
 C_6: =&  \{ y_j  w_j +y_{[j+1]_{2m+1}}\overline{w}_j , \overline{y}_j \overline{w}_j + \overline{y}_{[j+1]_{2m+1}} w_j \mid j = 1, 2, \ldots, 2m +1\}. \\
 C'_6: =&  \{ y_j \overline{w}_j , \overline{y}_j w_j  \mid j = 1, 2, \ldots, 2m +1\}. 
\end{align*}
A similar argument as computing $H(f,\{ 1,2\})$, we have $$H(f,\{ 1,3\})= \dim \Span _K (C_4 \cup C'_4 \cup C_5 \cup C'_5 \cup C'_6)=(4m+2)^2 = k^2.$$ Similarly, we can rewrite $f$ as follows 
\begin{align*}
f &=  \sum\limits_{1 \leq i,j \leq 2m+1, i: odd} x_iy_{[j-\frac{i-1}{2}]_{2m+1}}z_{[j+\frac{i-1}{2}]_{2m+1}}w_j 
   & + \sum\limits_{1 \leq i,j \leq 2m+1, i: odd} \overline{x}_i \overline{y}_{[j-\frac{i-1}{2}-m]_{2m+1}}z_{[j+\frac{i-1}{2}-m]_{2m+1}} \overline{w}_j \\
&+ \sum\limits_{1 \leq i,j \leq 2m+1, i: odd} x_i\overline{y}_{[j-\frac{i-1}{2}]_{2m+1}}\overline{z}_{[j+\frac{i-1}{2}]_{2m+1}} \overline{w}_j
&+ \sum\limits_{1 \leq i,j \leq 2m+1, i: odd} \overline{x}_iy_{[j-\frac{i-1}{2}-m]_{2m+1}} \overline{z}_{[j+\frac{i-1}{2}-m]_{2m+1}} w_j \\
& + \sum\limits_{1 \leq i,j \leq 2m+1, i: even} x_i\overline{y}_{[j-\frac{i-1}{2}]_{2m+1}}\overline{z}_{[j+\frac{i-1}{2}]_{2m+1}} \overline{w}_j
&+ \sum\limits_{1 \leq i,j \leq 2m+1, i: even} \overline{x}_iy_{[j-\frac{i-1}{2}-m]_{2m+1}}\overline{z}_{[j+\frac{i-1}{2}-m]_{2m+1}} w_j\\
&+ \sum\limits_{1 \leq i,j \leq 2m+1, i: even} x_iy_{[j-\frac{i-1}{2}]_{2m+1}}z_{[j+\frac{i-1}{2}]_{2m+1}} w_j
&+  \sum\limits_{1 \leq i,j \leq 2m+1, i: even} \overline{x}_i \overline{y}_{[j-\frac{i-1}{2}-m]_{2m+1}}z_{[j+\frac{i-1}{2}-m]_{2m+1}}\overline{w}_j\\
& + \sum\limits_{1 \leq j \leq 2m+1} \overline{x}_{2m+1} y_j \overline{z}_j w_j + \sum\limits_{1 \leq j \leq 2m+1}\overline{x}_{2m+1}\overline{y}_jz_j \overline{w}_j, 
\end{align*}
and we can see $H(f,\{ 1,4\}) = (4m+2)^2 = k^2.$ So, $f$ is a balanced neighborly polynomial of type $(k,k,k,k)$.
\end{proof}

\section{Balanced neighborly polynomials and balanced neighborly spheres}
%Conection to combbinatorics, Definition Balanced neibourly Gorenstein* complex; Stanley-Reisner rings.
In this section, we present a relation between balanced neighborly polynomials and balanced neighborly spheres. We first recall basic definitions on simplicial spheres. Let $\Delta$ be a finite simplicial complex on the vertex set $V=\{x_1,x_2,\ldots,x_n\}$.  Thus $\Delta$ is a nonempty collection of subsets of $V$ satisfying that $F \in  \Delta$ and $G \subset F$ imply $G \in \Delta.$ Elements of $\Delta$ are called {\it faces} of $\Delta$ and maximal faces (under inclusion) are called {\it facets.} The dimension of a face $F \in \Delta$ is $\dim F = |F|-1,$ and the {\it dimension of $\Delta$} is the maximum dimension of its faces.  Faces of dimension $0$ are called {\it vertices} and faces of dimension $1$ are called {\it edges.} Let $R=K[x_1, \ldots,x_n]$ be a polynomial ring over an infinite field $K.$ The ring $K[\Delta] = R/I_\Delta$, where $I_\Delta=(\underset{x_i \in F} {\prod} x_i \mid F \subset V, F \notin \Delta)$, is called {\it the Stanley--Reisner ring of $\Delta$.} It is well-known that the Krull dimension of $K[\Delta]$ equals $\dim \Delta + 1$ \cite[II Theorem 1.3]{St2}. If $\dim K[\Delta]=d,$ and a sequence $\Theta = \theta_1, \ldots, \theta_d$ of linear forms such that $\dim_K K[\Delta]/(\Theta) < \infty$ then $\Theta$ is called a {\it linear system of parameters} (l.s.o.p. for short) of $K[\Delta].$ 

A map $\kappa : V \to [d]$ is called {\it a proper $d$-coloring map} of $\Delta,$ if we have  $\kappa(i) \neq \kappa(j)$ for any edge $\{i,j \} \in  \Delta.$  We say that a $(d-1)$-dimensional simplicial complex $\Delta$ on $V$ is {\it balanced} (completely balanced in some literature) if its graph is $d$-colorable. Let $\Delta$ be a $d-1$ dimensional balanced simplicial complex and let $\kappa : V \to [d]$ be a proper $d$-coloring map of $\Delta.$ Set $V_k = \{ v \in V \mid \kappa(v) = k \}.$ Let $\theta_k = \underset {v \in V_k}  {\sum}x_v$ for $k = 1, 2, \ldots, d$. Then $\Theta$ is a l.s.o.p of $K[\Delta]$ \cite[III Proposition 4.3]{St2}. Recall that a {\it simplicial sphere} is a simplicial complex which is homeomorphic to a sphere. That means if $\Delta$ is a  simplicial sphere then $K[\Delta]$ is a Gorenstein ring and $K[\Delta]/(\Theta)$ is an Artinian Gorenstein algebra of socle degree $d$ \cite[ II Corollary 5.2]{St2}. A balanced simplicial sphere of dimension $d-1$ is called  {\it neighborly of type $(n_1, n_2, \ldots, n_d) \in \mathbb{Z}^d$} if $\dim_K (K[\Delta]/(\Theta))_{\e_S}=  \underset {k\in S} {\prod} n_k$ for all $S\subset [d]$.

From now on, we assume that $R=K[x_{i j} \mid 1 \leq i \leq d, 1 \leq j \leq n_i$] and $\Delta$ has vertex set $\{ x_{i j} \mid 1 \leq i \leq d, 1 \leq j \leq n_i\}$.
\begin{pro} \label{4.1}
If a balanced neighborly simplicial sphere of type $(n_1, \ldots, n_d)$ exists, then a balanced neighborly polynomial of type $(n_1, \ldots, n_d)$ exists over any field $K$.
\end{pro}
\begin{proof}
Set $R'=K[x_{i j} \mid 1 \leq i \leq d, 1 \leq j \leq n_i - 1$]. Then $R/(\Theta) \cong R'$, so 
$$R/ {(I_\Delta +(\Theta))} \cong R'/J \text { for some ideal } J.$$
By Proposition 4.3 in \cite{St1}, we have $x_{i j} x_{i k} = 0$ in $R/ (I_\Delta +(\Theta))$ for all $1\leq i \leq d$ and $1 \leq j \leq k \leq n_i$. It follows that $x_{i j} x_{i k} = 0$ in $S'/ J$ for all $1\leq i \leq d$ and $1 \leq j \leq k \leq n_i-1$. Since $R/(I_\Delta + \Theta)$ is Artinian Gorenstein, Lemma \ref{2.1} shows that there is a balanced polynomial $f\in R$ such that $\ann(f)= J$. This $f$ must be neighborly by Lemma \ref{2.2}, (i).
\end{proof}
By Proposition \ref{4.1} and Theorem \ref{3.1}, we can recover the following a result of Zheng \cite{Zh}, Proposition 6.
\begin{pro}
There are no balanced neighborly simplicial spheres of type $(2,2,2,2).$
\end{pro}
\medskip

\noindent {\bf Acknowledgement.} The authors would like to thank Professor Satoshi Murai for his interesting discussions  on this subject. The authors would like to thank the referee for the valuable comments to improve this article.

\end{document}